\documentclass[12pt]{amsart}
\usepackage[margin=20mm]{geometry}

\usepackage[normalem]{ulem}

\usepackage[T1]{fontenc}
\usepackage[latin1]{inputenc}
\usepackage{amsmath,amssymb}
\usepackage{mathtools}
\usepackage{hyperref}
\usepackage{ae}
\usepackage{color}
\usepackage{comment}
\usepackage{bm}
\usepackage{graphicx}

\usepackage{color}

\newcommand{\N}{\mathbb{N}}

\newcommand{\R}{\mathbb{R}}

\newcommand{\U}{\mathbb{U}}

\DeclareMathOperator{\e}{\mathrm{e}}

\newcommand{\norm}[1]{\left\| #1 \right\|}

\newtheorem{theorem}{Theorem}
\newtheorem{proposition}{Proposition}
\newtheorem{corollary}{Corollary}
\newtheorem{lemma}{Lemma}

\begin{document}
\title{M\"obius orthogonality of sequences with maximal entropy}

\author{Michael Drmota}
\email{michael.drmota@tuwien.ac.at}
\address{Institut f\"ur Diskrete Mathematik und Geometrie
TU Wien\\
Wiedner Hauptstr. 8--10\\
1040 Wien, Austria}
\author{Christian Mauduit}
\author{Jo\"el Rivat}
\email{joel.rivat@univ-amu.fr}
\address{Universit\'e d'Aix-Marseille\\
Institut de Math\'ematiques de Marseille\\
CNRS UMR 7373\\
163, avenue de Luminy, Case 907\\
13288 MARSEILLE Cedex 9, France}
\author{Lukas Spiegelhofer}
\email{lukas.spiegelhofer@tuwien.ac.at}
\address{Institut f\"ur Diskrete Mathematik und Geometrie
TU Wien\\
Wiedner Hauptstr. 8--10\\
1040 Wien, Austria}
%

\thanks{The first and fourth author is supported by the
Austrian Science Foundation FWF, SFB F55-02 ``Subsequences of Automatic Sequences and Uniform Distribution''.
This work was supported by the joint ANR-FWF-project  ANR-14-CE34-0009, I-1751 MuDeRa.}

\begin{abstract}
We prove that strongly $b$-multiplicative functions of modulus $1$ along squares are asymptotically orthogonal to the M\"obius function.
This provides examples of sequences having maximal entropy and satisfying this property.
\end{abstract}

\maketitle

\section{Introduction}
Sarnak's conjecture~\cite{S2015,S2012} is concerned with the M\"obius $\mu$-function, defined by $\mu(n)=(-1)^{\omega(n)}$ is $n$ is squarefree, and $\mu(n)=0$ otherwise, where $\omega(n)$ is the number of different prime factors of $n$.
It can also be defined as the Dirichlet inverse of the constant function $1$.
Sarnak's conjecture states that every bounded \emph{deterministic} sequence $f:\mathbb N\rightarrow \mathbb C$ is orthogonal to the M\"obius function,
\[\sum_{n<N}\mu(n)f(n)=o(N).\]
Deterministic sequences $f$ can be defined by the property that for all $\varepsilon>0$, the set of $k$-tuples 
\[\bigl\{\bigl(f(n+0),\ldots,f(n+k-1)\bigr):n\geq 0\bigr\}\subseteq\mathbb C^k\]
can be covered by $\exp(o(k))$ many balls of radius $\varepsilon$, as $k$ goes to infinity. 
For functions $f$ having values in a finite set, it is equivalent to demand that $f$ has subexponential \emph{factor complexity} $p_k$: the number of contiguous finite subsequences of $f$ of length $k$ should be bounded by $\exp(o(k))$.
For example, this is the case for all
\emph{automatic sequences}~\cite{AS2003},
which have a factor complexity bounded by $C k$
(where $C>0$ is a constant depending on the sequence),
and Sarnak's conjecture has been verified for this class of sequences by M\"ullner~\cite{M2017}.
Sarnak's conjecture has been verified for other classes of sequences, see for example~\cite{Bourgain2013a, Bourgain2013, BSZ2013,
Davenport1937,DK2015, FM2015, Green2012, GT2012, EKL2016, HKLD2014, ELD2014, HLD2017, HLD2017b, Karagulyan2015, KL2015, LS2015, Peckner2015, SU2015, Veech2016}.

In this work, we are concerned with M\"obius orthogonality for \emph{non-deterministic} sequences --- in particular, we are concerned with the \emph{normal sequence} $\mathbf t(n^2)$, where $\mathbf t$ is the \emph{Thue--Morse sequence}.

It is known that there exist (many) normal sequences that are M\"obius disjoint;
in particular, \emph{each} measure-theoretic dynamical system $(X,\mathcal B, \lambda,T)$ is almost everywhere M\"obius orthogonal: for each $f\in L^1(X)$ we have
\begin{equation}\label{eqn_dynamical_system}
\lim_{N\to\infty}\frac1N\sum_{n\leq N} f(T^nx)\mu(n)=0
\end{equation}
for almost all $x\in X$
\cite{S2015}; see~\cite{HKLD2014} for a proof and~\cite{E2015} for a polynomial extension.
Considering the Bernoulli shift on $\{0,1\}^{\mathbb Z}$ we obtain many normal sequences with the desired property.
Moreover, M\"obius orthogonality for dynamical systems having large positive entropy (close to the maximal value) was considered recently by Downarowicz and Serafin~\cite{DS2019,DS2020}.
The added value of our paper lies in an explicit, simple construction of a normal number that is M\"obius orthogonal.
We thank Mariusz Lema\'nczyk for pointing out this remark to us.

The Thue--Morse sequence can be defined via the binary sum-of-digits function $s_2$, which counts the number of powers of two needed to represent a natural number as their sum.
We define $\mathbf t(n)=(-1)^{s_2(n)}$, which is the Thue--Morse sequence on the two symbols $1,-1$.
This sequence is automatic and as such has factor complexity $p_k\leq Ck$;
however, when we extract the subsequence along the squares, the resulting sequence is normal. That is, each finite word of length $k$ on $\{1,-1\}$ occurs with asymptotic frequency $2^{-k}$ along this subsequence.
This has been proved by the first three authors~\cite{DMR2019}, strengthening a result of Moshe~\cite{M2007}, who showed that each block $b\in\{1,-1\}^k$ occurs at least once in $\mathbf t(n^2)$.

Besides providing an example of a sequence having maximal topological entropy and being orthogonal to $\mu$,
our interest in the sum $\sum_{n<N}\mu(n)\mathbf t(n^2)$ has its origin in the study of the digits of prime numbers.

The second and third authors~\cite{MR2010} proved in particular that the base-$b$ sum-of-digits of prime numbers is uniformly distributed in residue classes;
this was accomplished by studying the sum $\sum_{n<N}\Lambda(n)\exp(2\pi i\vartheta s_b(n))$,
where $\Lambda$ is the von Mangoldt function 
(defined by $\Lambda(n)=\log p$ if $n=p^k$ with $k\in\N$, $k\geq 1$ 
and $\Lambda(n) = 0$ otherwise).
Moreover, the same authors~\cite{MR2009} studied the sum of digits of the sequence of squares.
It is therefore a natural problem to attack the sum of digits of \emph{squares of primes}; for the Thue--Morse sequence, this can be accomplished by studying the sum $\sum_{n<N}\Lambda(n)\mathbf t(n^2)$.
For the time being, we do not have a solution for this problem; a replacement is the (easier) sum $\sum_{n<N}\mu(n)\mathbf t(n^2)$.

\subsection{Notation}
We denote by $\N$ the set of non-negative integers, 
and by $\U$ the set of complex numbers of modulus~$1$.
For $n\in\N$, $n\geq 1$, we denote by 
$\tau(n)$ the number of divisors of $n$, by 
$\omega(n)$ the number of distinct prime factors of $n$, by 
and by $\mu(n)$ the M\"obius function
(defined by $\mu(n)= (-1)^{\omega(n)}$ if $n$ is squarefree and
$\mu(n)=0$ otherwise).

For $x\in\R$ we denote 
by $\pi(x)$ the number of prime numbers less or equal to $x$,
by $\norm{x}$ the distance of $x$ to the nearest integer,
and we set $\e(x) = \exp(2i\pi x)$.
If $f$ and $g$ are two functions taking strictly positive values such
that $f/g$ is bounded, we write $f=O(g)$ or $f\ll g$.

Furthermore let $s_2(n)$ denote the binary sum-of-digits function.

\subsection{Main Result}

Let $t(n) = s_2(n) \bmod 2$ denote the Thue--Morse sequence on the alphabet $\{0,1\}$.
It has been shown by the three first authors that subsequence $t(n^2)$ is a normal sequence, that is,
each binary block $B\in \{0,1\}^L$, $L\ge 1$, appears with asymptotic frequency $2^{-L}$ as
a factor in $t(n^2)$. In particular this shows that $t(n^2)$ has maximal (positive) entropy $\log 2$.

The main purpose of this paper is to show that $t(n^2)$ is orthogonal to the M\"obius function.
\begin{theorem}\label{Th1}
Let $t(n)$ denote the Thue--Morse sequence. Then we have, as $N\to\infty$,
\begin{equation}\label{eqTh1}
\sum_{n< N} \mu(n) t(n^2) = o(N).
\end{equation}
\end{theorem}
This is actually not the first explicit example of a positive entropy sequence that is orthogonal to the M\"obius function.
A previously considered example is given by the sequence $\mu(n)^2$ (detecting the square-free integers), which has topological  entropy $\frac 6{\pi^2} \log 2$(\cite{P2015,S2015}, see also~\cite{DKKL2018,ALR2015}) and obviously $\mu(n)^2$ is orthogonal to $\mu(n)$.

Nevertheless, our result is one of the first explicit examples of a (binary) sequence with maximal entropy $\log 2$ that has this orthogonality property.

This kind of examples is in particular interesting in view of the 
Sarnak conjecture~\cite{S2015,S2012} which says that every bounded zero entropy sequence is orthogonal to
the M\"obius function. 

In this context, we note that due to normality, the \emph{symbolic dynamical system} $(X,\mathcal B,\lambda,T)$ defined by $t(n^2)$ is the full shift:
clearly, there exist sequences $x\in X$ that are not orthogonal to the M\"obius function.
On the other hand, note that Downarowicz and Serafin~\cite{DS2019,DS2020} study dynamical systems with entropy close to the maximum and still obtain M\"obius orthogonality.
\medskip

The Thue--Morse sequence is sometimes defined by $g(n)=(-1)^{s_2(n)}$.
That is, the values $0,1$ are
replaced by $1$ and $-1$. Since $g(n) = 1- 2t(n)$ and $\sum_{n< N} \mu(n) = o(N)$ the relation
(\ref{eqTh1}) is equivalent to 
\begin{equation}\label{eqTh1-2}
\sum_{n< N} \mu(n) (-1)^{s_2(n)}  = o(N).
\end{equation}
The function $g(n)=(-1)^{s_2(n)}$ is a so-called {\it strongly $2$-multiplicative function}. More
generally, a strongly $b$-multiplicative functions (where $b\ge 2$ is a fixed integer) is defined by the relation
\[
g(kb + a) = g(k) g(a)\qquad (a,k\in \mathbb{N},\, 0\le a < b).
\]
Actually, Theorem~\ref{Th1} can be generalized to all complex valued strongly $b$-multiplicative functions
of modulus $1$.
\begin{theorem}\label{Th2}
Let $b\ge 2$ be a given integer. Then for all complex valued strongly $b$-multiplicative functions $g(n)$ of modulus 
$1$ we have
\begin{equation}\label{eqTh2}
\sum_{n< N} \mu(n) g(n^2) = o(N).
\end{equation}
\end{theorem}
Note that this theorem gives many more examples of sequences with maximal entropy that are orthogonal to M\"obius:
M\"ullner~\cite{M2018} proved in particular that $q$-multiplicative functions with values in $\{\exp(2\pi i j/m):0\leq j<m\}$ are normal along the squares under certain weak conditions.

%

\subsection{Strongly $b$-Multiplicative Functions}

It is clear that strongly $b$-multiplicative functions $g(n)$ of modulus $1$ satisfy
$g(0) = 1$ and that $g(1),\ldots,g(b-1)$ determine all other values of $g(n)$:
\[
g(n) = \prod_{j\ge 0} g(\varepsilon_j)  \quad \mbox{with} \quad
n = \sum_{j\ge 0} \varepsilon_j b^j.
\]

We will distinguish between two different classes of $b$-multiplicative functions, namely periodic ones and non-periodic ones.
\begin{proposition}
A $b$-multiplicative function $g$ having values in $\{z\in\mathbb C:\lvert z\rvert=1\}$ is periodic if and only if
\begin{equation}\label{eqn_periodic}
g(\ell) = g(1)^\ell \quad (0\le \ell \le b-1) \quad\mbox{and}\quad  g(b-1) = 1.
\end{equation}
\end{proposition}
While the difficult part of the proof (the ``only if''-part) of this statement rests on Proposition~\ref{prp_fourierprop} proved later, we will not use this direction in the sequel and thus there is no circular argument involved.
\begin{proof}
Suppose first that~\eqref{eqn_periodic} holds, that is, $g(\ell) = \e(\ell j_0/(b-1))$ for some integer $j_0$.
Then
\[    g(n) = \e(n j_0/(b-1))    \]
for all $n\ge 0$. This follows from the fact that $\e(b^j/(b-1)) = \e(1/(b-1))$.
This means that in this case $g(n)$ is periodic with a period dividing $b-1$.
Conversely, suppose, in order to obtain a contradiction, that~\eqref{eqn_periodic} is not satisfied and that $g$ is periodic with period $L$.
By Proposition~\ref{prp_fourierprop} below we have
\begin{equation}\label{eqn_fourier_convergence}
F_\lambda(h)=o(1)
\end{equation}
as $\lambda\rightarrow\infty$, for all $h\in\mathbb Z$.
By periodicity,
\[F_\lambda(h)=\frac 1{b^\lambda}\sum_{0\leq u<b^\lambda}g(u)\e(-hu/L)=\mathcal O(L/b^\lambda)+\frac 1L\sum_{0\leq u<L}g(u)\e(-hu/L)\]
and~\eqref{eqn_fourier_convergence} implies that $\sum_{0\leq u<L}g(u)\e(-hu/L)=0$ for $0\leq h<L$.
By inversion, we obtain $g(u)=0$ for all $u$, which contradicts $\lvert g(u)\rvert=1$. This completes the proof.
\end{proof}

\subsection{Plan of the Proofs}

If $g(n)$ is periodic then Dirichlet's prime number theorem implies Theorem~\ref{Th2}.
Hence, it is sufficient to suppose that $g(n)$ is not periodic.

In order to prove Theorem~\ref{Th2} we apply the Daboussi--K\'atai
criterion (Lemma~\ref{lemma:daboussi-katai} below). 
This criterion says that 
\begin{equation}\label{eqs2s2}
\sum_{n<N} g(p^2 n^2) \overline{ g(q^2 n^2) }  = o(N),
\end{equation}
where $p,q$ are different (and sufficiently large) prime numbers, implies Theorem~\ref{Th2}.

At this stage we will apply a general theorem by the second and third authors~\cite{MR2018} that
gives sufficient conditions for functions $f(n)$ (with $|f(n)|\le 1$) such that
\[
\sum_{n< N} f(n^2)\e(\theta n) = o(N).
\]
In our case we want to apply this theorem for 
\[
f(n) = g(p^2 n) \overline{ g(q^2 n) }
\]
and $\theta = 0$. In particular one has to check a {\it carry property} and
a {\it Fourier property}. In our case the carry property is easy to check (see Section~\ref{secTh2}),
whereas the Fourier property needs non-trivial bounds for the Fourier-terms
\[
F_\lambda(t) = \frac 1{b^\lambda} \sum_{0\le u < b^\lambda} f(u) \e(-ut) = 
\frac 1{b^\lambda} \sum_{0\le u < b^\lambda}  g(p^2 u) \overline{ g(q^2 u) } \e(-ut),
\] 

We will derive the necessary bounds in Section~\ref{sec:Fourier}.
This will be then the main ingredient for the proof of Theorem~\ref{Th2}
which will be summarized in Section~\ref{secTh2}.

%

\section{Fourier bounds}\label{sec:Fourier}
In this section, we are concerned with strongly $b$-multiplicative functions $g:\mathbb N\rightarrow\U$ that do not satisfy~\eqref{eqn_periodic}.
In the main result of this section, Proposition~\ref{prp_fourierprop} below, we will prove that they possess Fourier coefficients $F_\lambda(t)$ that converge to zero uniformly in $t$.

We suppose that $P,Q$ are positive and coprime integers that are also coprime to $b$ --- later we will apply our
results for $P = p^2$ and $Q = q^2$, where $p,q$ are different primes.
In order to obtain upper bounds for $F_\lambda(t)$ we define more generally 
\[
F_\lambda^{i,j}(t) = \frac 1{b^\lambda} \sum_{0\le u < b^\lambda} g(P u+i) \overline{ g(Qu+j) } \e(-ut),
\] 
where $0\le i \le P-1$ and $0\le j \le Q-1$. 

The following recurrence follows directly from the definition.
\begin{lemma}\label{Le1}
Suppose that $P,Q$ are positive and coprime integers that are also coprime to $b$ and
 that $0\le i \le P-1$, $0\le j \le Q-1$, and $\lambda \ge 1$. Then we have for all $t\in \R$
\begin{equation}\label{eqLe1}
F_\lambda^{i,j}(t) = \frac 1b \sum_{r=0}^{b-1}
g(Pr+i \bmod b) \overline{ g(Qr+j \bmod b) } \e(-rt) \,
F_{\lambda-1}^{\left\lfloor \frac {i+rP}b \right\rfloor, \left\lfloor \frac {j+Qr}b \right\rfloor }(bt).
\end{equation}
\end{lemma}

\begin{proof}
By distinguishing  between residue classes modulo $b$ we obtain
\begin{align*}
F_\lambda^{i,j}(t) &= \frac 1{b^\lambda} \sum_{r=0}^{b-1} \sum_{0\le u < b^{\lambda-1}}
g(P(bu+r)+i) \overline{ g(Q(bu+r)+j) } \e(-(bu+r)t)   \\
&= \frac 1{b^\lambda} \sum_{r=0}^{b-1} 
g(Pr+i \bmod b) \overline{ g(Qr+j \bmod b) } \e(-rt) \\
&\qquad \qquad \times 
\sum_{0\le u < b^{\lambda-1}} g(bPu + b \lfloor (Pr+i)/b \rfloor ) \overline{g(bQu + b \lfloor (Qr+j)/b \rfloor ) } \e(-but) \\
&= \frac 1b \sum_{r=0}^{b-1}
g(Pr+i \bmod b) \overline{ g(Qr+j \bmod b) } \e(-rt) 
F_{\lambda-1}^{\left\lfloor \frac {i+rP}b \right\rfloor, \left\lfloor \frac {j+Qr}b \right\rfloor }(bt)
.
\end{align*}
\end{proof}

Actually we are interested in the behaviour of $F_\lambda^{0,0}(t) = F_\lambda(t)$. Thus, we have
to study the action 
\[
T: (i,j) \to \left\{ \left( \left\lfloor \frac ib \right\rfloor , \left\lfloor \frac jb \right\rfloor \right),
\left( \left\lfloor \frac {i+P}b \right\rfloor , \left\lfloor \frac {j+Q}b \right\rfloor \right),
\ldots, 
\left( \left\lfloor \frac {i+(b-1)P}b \right\rfloor , \left\lfloor \frac {j+(b-1)Q}b \right\rfloor \right)
 \right\},
\]
where we start with $(0,0)$. In this context it is convenient to consider the di-graph $D$ with vertices 
$(i,j)$ ($0\le i \le P-1$, $0\le j \le Q-1$) and edges 
\[
(i,j) \to \left( \left\lfloor \frac ib \right\rfloor , \left\lfloor \frac jb \right\rfloor \right), \ 
(i,j) \to \left( \left\lfloor \frac {i+P}b \right\rfloor , \left\lfloor \frac {j+Q}b \right\rfloor \right),
\ldots, 
(i,j) \to \left( \left\lfloor \frac {i+(b-1)P}b \right\rfloor , \left\lfloor \frac {j+(b-1)Q}b \right\rfloor \right).
\]

\begin{lemma}\label{Le2}
Let $\mathcal{C}$ denote the strongly connected component of the di-graph $D$ that contains $(0,0)$.
Then $\mathcal{C}$ contains precisely $P+Q-1$ elements that can be also represented by
\[
\mathcal{C} = \left\{ \left( \lfloor tP \rfloor, \lfloor tQ \rfloor  \right): 0\le t < 1 \right\}.
\]
In particular if $(i,j)\in \mathcal{C}$ and $(i,j)\ne (P-1,Q-1)$ then either $(i+1,j)\in \mathcal{C}$ or 
$(i,j+1)\in \mathcal{C}$.
\end{lemma} 

\begin{proof}
Clearly we have $(0,0)\in \left\{ \left( \lfloor tP \rfloor, \lfloor tQ \rfloor  \right): 0\le t < 1 \right\}$;
we just have to set $t=0$.

Next we show that for $0\le t < 1$ and for integers $0\le r \le b-1$
\begin{equation}\label{eqidentity}
\left\lfloor \frac{\lfloor tP \rfloor + rP }b \right\rfloor = \left\lfloor \frac{t+r}b P \right\rfloor.
\end{equation}
For this purpose we write $tP = \lfloor tP \rfloor + y$ with $0\le y < 1$. This gives
\[
\frac{\lfloor tP \rfloor + rP }b = \frac{t+r}b P - \frac yb.
\]
Suppose now that for some integer $m$ 
\begin{equation}\label{eqinterval}
m \le \frac{t+r}b P < m + \frac 1b.
\end{equation}
Equivalently this means that
\[
bm-rP \le t < bm-rP+1.
\]
Since $0\le t < 1$ this implies that $bm = rP$. If $1\le r \le b-1$ this is impossible since 
$b$ and $P$ are coprime. Thus we either have $r=m=0$, that is, $tP < 1$ or (\ref{eqinterval}) 
does not hold. In the first case we have $\lfloor tP \rfloor = 0$ (or $y = \lfloor tP \rfloor$)
and consequently (\ref{eqidentity}) is just the trivial identity $0 = 0$ (note that $r=0$).
In the second case (where (\ref{eqinterval}) does not hold) we clearly have
\[
\left\lfloor \frac{t+r}b P - \frac yb \right\rfloor = \left\lfloor \frac{t+r}b P \right\rfloor
\]
so that (\ref{eqidentity}) holds, too.

Clearly (\ref{eqidentity}) remains true if we replace $P$ by $Q$. Hence, if $(i,j)$ is represented
by $(i,j) = \left( \lfloor tP \rfloor, \lfloor tQ \rfloor  \right)$ then for every $0\le r \le b-1$
we have
\[
\left( \left\lfloor \frac {i+rP}b \right\rfloor , \left\lfloor \frac {j+rQ}b \right\rfloor \right)
= \left( \left\lfloor  \frac{t+r}b P  \right\rfloor , \left\lfloor \frac{t+r}b Q  \right\rfloor \right).
\]
Note that $0\le \frac{t+r}b < 1$ so that we stay in the same set. Furthermore if we start with $(0,0)$
represented by $t=0$ then by repeated application it follows that we can reach any pair of the kind
\[
(i,j) = \left( \left\lfloor  \frac{r_L + r_{L-1} b + \cdots r_{1}b^{L-1}}{b^L} P  \right\rfloor , 
\left\lfloor \frac{r_L + r_{L-1} b + \cdots r_{1}b^{L-1}}{b^L} Q  \right\rfloor \right).
\]
Actually this is sufficient to reach all elements of 
$\left\{ \left( \lfloor tP \rfloor, \lfloor tQ \rfloor  \right): 0\le t < 1 \right\}$.
Since $P$ and $Q$ are coprime the line $\{(tP,tQ): 0\le t < 1\}$ does not meet a lattice point 
different from $(0,0)$. Consequently line is cut into $P+Q-1$ intervals that correspond to 
its $P+Q-1$ elements. In each of this interval we could restrict ourselves to $b$-adic rational numbers 
$t$. This means that starting with $(0,0)$ we can reach every element of 
$\left\{ \left( \lfloor tP \rfloor, \lfloor tQ \rfloor  \right): 0\le t < 1 \right\}$.
Conversely if we start with the pair $\left( \lfloor tP \rfloor, \lfloor tQ \rfloor  \right)$ and if
$L$ is large enough then 
\[
\left( \left\lfloor \frac t{b^L} P  \right\rfloor , \left\lfloor \frac t{b^L} Q   \right\rfloor \right) = (0,0).
\]
Summing up this means that we have actually described the strongly connected component 
of the di-graph $D$ that contains $(0,0)$.
\end{proof}

As a corollary we obtain the following property that will be crucial for the proof of a non-trivial
upper bound of $F_\lambda^{0,0}(t)$.

\begin{corollary}
Let $\mathcal{C}$ be as above and assume that $b<P< Q$. Then there exists $i_0 < b$ such
that $\{(i_0,b-1),(i_0,b)\} \subseteq \mathcal{C}$.
\end{corollary}

\begin{proof}
By Lemma~\ref{Le2}, $\mathcal{C}$ can be considered as a lattice {\it path} from $(0,0)$ to $(P-1,Q-1)$
that is close to the diagonal and has only steps of the form $(i,j)\to (i+1,j)$ and  $(i,j)\to (i,j+1)$.
Thus, there is a unique step of the form $(i_0,b-1)\to(i_0,b)$. Since $P<Q$ it follows that $i_0\le b-1$.
\end{proof}

Next we use the relation (\ref{eqLe1}) to obtain proper vector recurrences for $F_\lambda^{i,j}(t)$.
Set 
\[
{\bf F}_\lambda(t) = \left( F_\lambda^{i,j}(t) \right)_{(i,j)\in \mathcal{C}}
\]
and 
\[
{\bf A}(t) = \left( a_{(i,j),(i',j')}(t) \right)_{(i,j),(i',j')\in \mathcal{C}}
\]
where
\[
a_{(i,j),(i',j')}(t) = \left\{ \begin{array}{cl} 
\frac 1b g(Pr+i \bmod b) \overline{ g(Qr+j \bmod b) } \e(-rt) & 
\mbox{for $(i',j') = \left( \left\lfloor \frac {i+rP}b \right\rfloor , \left\lfloor \frac {j+rQ}b \right\rfloor \right)$,}\\
0 &\mbox{else.}
\end{array} \right.
\]
Then (\ref{eqLe1}) rewrites to 
\[
{\bf F}_\lambda(t) =  {\bf A}(t) \cdot {\bf F}_{\lambda-1}(bt).
\]
Thus, we are led to study the product of matrices ${\bf A}(t)\cdot {\bf A}(bt)\cdots  {\bf A}(b^Lt)$.

Let $\| \cdot \|$ denote that row-sum-norm of a matrix. Then we have the following property.
\begin{lemma}\label{Le3}
Suppose that $g(n)$ is non-periodic.
There exist $L > 0$ and $\delta > 0$ such that 
\begin{equation}\label{eqLe3}
\sup_{t\in \R} \| {\bf A}(t)\cdot {\bf A}(bt)\cdots  {\bf A}(b^Lt) \| \le 1 - \delta.
\end{equation}
\end{lemma}

\begin{proof}
We interpret the entries of the matrix ${\bf A}(t)\cdot {\bf A}(bt)\cdots  {\bf A}(b^Lt)$
in the following way. Let $D_C(t)$ be the strongly connected subgraph of $D$ corresponding to the
vertex set $\mathcal{C}$, where the edges of $D_C(t)$ 
\[
(i,j) \to \left( \left\lfloor \frac {i+rP}b \right\rfloor , \left\lfloor \frac {j+rQ}b \right\rfloor \right), \qquad
0\le r \le b-1,
\]
are labelled by  
\[
a_{(i,j),(\langle (i+rP)/b \rangle,\langle (j+rQ)/b \rangle )}(t)= \frac 1b g(Pr+i \bmod b) \overline{ g(Qr+j \bmod b) } \e(-rt).
\]
If $(e_0,e_1,\ldots,e_L)$ be a directed path in $D$ such that $e_j$ is actually an edge in $D_C(b^jt)$, $0\le j \le L$,
then we define the weight $w$ of this path by
\[
w(e_0,e_1,\ldots,e_L) = a_{e_0}(t) a_{e_1}(bt)\cdots a_{e_L}(b^Lt).
\]
Note that $|w(e_0,e_1,\ldots,e_L)| = b^{-L-1}$.
It the entries of ${\bf A}(t)\cdot {\bf A}(bt)\cdots  {\bf A}(b^Lt)$ are denoted by
$b_{L+1;(i,j),(i'j')}(t)$ then we have by definition
\[
b_{L+1;(i,j),(i'j')}(t) = \sum w(e_0,e_1,\ldots,e_L),
\]
where the sum is taken over all directed paths $(e_0,e_1,\ldots,e_L)$ in $D$ that connect $(i,j)$ and $(i',j')$
such that $e_j$ is an edge in $D_C(b^jt)$, $0\le j \le L$.

For $(i,j),(i',j')\in \mathcal{C}$ let $B_{L+1}(i,j),(i'j')$ denote the number of different paths
from $(i,j)$ to $(i',j')$. Clearly we have 
\[
\sum_{(i',j')\in \mathcal{C}}
B_{L+1}((i,j),(i'j')) = b^{L+1}.
\]
Hence
\[
\sum_{(i',j')\in \mathcal{C}} \left| b_{(i,j),(i'j')}(t) \right| \le b^{L-1} \sum_{(i',j')\in \mathcal{C}}
B_{L+1}((i,j),(i'j')) = 1. 
\]
Note that this just says that $\| {\bf A}(t)\cdot {\bf A}(bt)\cdots  {\bf A}(b^Lt) \|  \le 1$
Furthermore, in order to prove (\ref{eqLe3}) we just have to show that
for every $(i,j)\in \mathcal{C}$ there exists $(i',j')\in \mathcal{C}$ with
\begin{equation}\label{eqgoal}
\left| b_{L+1;(i,j),(i'j')}(t) \right| <  b^{-L-1} B_{L+1}((i,j),(i'j')).
\end{equation}

In order to prove (\ref{eqgoal}) we proceed in two steps. 
We first show that there exist $L\ge 1$ such that
\begin{equation}\label{eqgoal1}
\left| b_{L+1;(i_0,b-1),(0,0)}(t) \right| 
<  b^{-L-1} B_{L+1}((i_0,b-1),(0,0))
\end{equation}
or 
\begin{equation}\label{eqgoal2}
\left| b_{(i_0,b),(0,0)}(t) \right| 
<  b^{-L-1} B_{L+1}((i_0,b),(0,0)).
\end{equation}

Since $D_C(t)$ is strongly connected it is clear that these properties imply (\ref{eqgoal})
for some $L> 0$.
We first define define $L_1$ the minimal $n$ such that for every pair $((i,j),(i',j'))\in \mathcal{C}$
there exists a path of length $L_1$ that connects $(i,j)$ and $(i',j')$. (Since there is s loop
from $(0,0)$ to itself, there are such $n$.) Second we define $L_2$ as the smallest $L$ such that
the above construction works. Then for every $(i,j)\in\mathcal{C}$ there are two paths $p_1,p_2$ of length $L_1$
that connect $(i,j)$ to $(0,1)$ and $(i,j)$ to $(0,2)$, respectively. 
This shows that $B_{L_1}((i,j),(0,1))>0$ and $B_{L_1}((i,j),(0,2))>0$.
Consequently we have 
\begin{align*}
\left| b_{L_1+L_2+1;(i,j),(0,0)}(t) \right| &= 
\left| \sum_{(i',j')\in \mathcal{C}} b_{L_1;(i,j),(i',j')}(t) b_{L_2+1;(i',j'),(0,0)}(b^{L_1}t) \right| \\
&< b^{-L_1-L_2-1} \sum_{(i',j')\in \mathcal{C}}  B_{L_1}((i,j),(i',j'))B_{L_2+1}((i',j'),(0,0)) \\
&=  b^{-L_1-L_2-1}  B_{L_1+L_1+1}((i,j),(0,1))
\end{align*}
since $(i',j') = (i_0,b-1)$ or $(i',j') = (i_0,b)$ appears in this sum with a non-zero contribution,
and so (\ref{eqgoal}) follows.

We fix some $1\le r \le b-1$ and consider two paths from $(i_0,b-1)$ to $(0,0)$ and
two from $(i_0,b)$ to $(0,0)$, respectively:
\begin{align*}
&(i_0,b-1) \to \left( \left\lfloor \frac{i_0+rP}b  \right\rfloor, \left\lfloor \frac{rQ+b-1}b \right\rfloor \right) \to  
\left( \left\lfloor \frac{i_0+rP}{b^2} \right\rfloor, \left\lfloor  \frac{rQ+b-1}{b^2}   \right\rfloor \right) \to \cdots\\
& \qquad \qquad  \to
\left( \left\lfloor \frac{i_0+rP}{b^L} \right\rfloor, \left\lfloor \frac{rQ+b-1}{b^L} \right\rfloor \right) = (0,0),\\
&(i_0,b-1) \to (0,0) \to \cdots \to (0,0), \\
&(i_0,b)\to \left( \left\lfloor \frac{i_0+rP}b  \right\rfloor, \left\lfloor \frac{rQ+b}b \right\rfloor \right) = 
\left( \left\lfloor \frac{i_0+rP}b  \right\rfloor, \left\lfloor \frac{rQ+b-1}b \right\rfloor \right) \to  \\ 
&\qquad \qquad 
\left( \left\lfloor \frac{i_0+rP}{b^2} \right\rfloor, \left\lfloor  \frac{rQ+b-1}{b^2}   \right\rfloor \right) \to \cdots \to
\left( \left\lfloor \frac{i_0+rP}{b^L} \right\rfloor, \left\lfloor \frac{rQ+b-1}{b^L} \right\rfloor \right) = (0,0) ,\\
&(i_0,b) \to (0,1) \to (0,0) \to \cdots \to (0,0),
\end{align*}
where $L$ is chosen in a way that $b^{L+1} > \max\{P+1,Q+1\}$.
Here we have used the facts that (since by assumption $rQ/b$ is not an integer)
\[
\left\lfloor \frac{rQ+b-1}b \right\rfloor = \left\lfloor \frac{rQ+b}b \right\rfloor
\]
and that for all non-negative integers $a$
\[
\left\lfloor \frac{\lfloor a/b \rfloor }b \right\rfloor = \left\lfloor \frac{a}{b^2} \right\rfloor.
\]

The weights of these paths are given by
\[
v_1(r):=g(i_0+rP \bmod b) \overline { g(rQ-1 \bmod b ) } \e(-rt) A, \quad w_1:=g(i_0) \overline{g(b-1) } 
\]
and by
\[
v_2(r):=g(i_0+rP \bmod b) \overline { g(rQ\bmod b ) } \e(-rt) A, \quad
w_2:= g(i_0) \overline{g(1) } 
\]
where
\[
A = \prod_{j=1}^{L+1} g\left( \lfloor (i_0 + rP)/b^j \rfloor \right) \overline {  g\left( \lfloor (rQ+b-1)/b^j \rfloor \right) }.
\]
Thus we have
\[
\left| b_{(i_0,b-1),(0,0)}(t) \right| \le b^{-L-1} \left( B_{L+1}((i_0,b-1),(0,0))-2+ |v_1(r) + w_1| \right)
\]
and
\[
\left| b_{(i_0,1),(0,0)}(t) \right| \le b^{-L-1} \left( B_{L+1}((i_0,b),(0,0))-2+ |v_2(r)+w_2| \right)
\]
Thus, in order to prove Lemma~\ref{Le3} we just have to check that there exist $1\le r \le b-1$ with
\begin{equation}\label{eqtoshow}
\min\{ |v_1(r) + w_1|, |v_2(r) + w_2| \} < 2.
\end{equation}

Suppose that the converse statement holds, that is, for all $1\le r \le b-1$ we have
\[
|v_1(r) + w_1| = |v_2(r) + w_2| = 2.
\]
Then we would have (for all $1\le r \le b-1$)
\[
v_1(r)/w_1 = v_2(r)/w_2
\]
or equivalently 
\[
g(rQ \bmod b) = g(rQ-1 \bmod b) g(1) \overline { g(b-1) }.
\]
Since $rQ \bmod b$, $1\le r\le b-1$, runs precisely through the residue classes $1\le \ell \le b-1$ this implies
\[
g(\ell) = g(\ell-1) g(1) \overline { g(b-1) }.
\]
By setting $\ell = 1$ it follows that $g(b-1) = 1$ (since $g(0)=1$) and consequently we have
\[
g(\ell) = g(1)^\ell, \qquad 1\le \ell \le b-1.
\]
Recall that $g(b-1) = 1$. Hence, we have $g(\ell) = \e(\ell j_0/(b-1))$ for some $j_0$ and consequently
$g(n)$ is periodic. This is of course a contradiction and so (\ref {eqtoshow}) (and consequently Lemma~\ref{Le3}) follows.
\end{proof}

This finally implies the main result of this section.

\begin{proposition}\label{prp_fourierprop}
There exist constants $C > 0$ and $\eta > 0$ such that 
for all $\lambda\ge 0$
\[
\sup_{t\in \R} \left| F_\lambda(t) \right| \le C\, e^{-\eta \lambda}.
\]
\end{proposition}

\begin{proof}
It follows from Lemma~\ref{Le3} that 
\[
\| {\bf A}(t)\cdot {\bf A}(bt)\cdots  {\bf A}(b^\lambda t) \| \le (1-\delta)^{\lfloor \lambda/(L+1) \rfloor} 
\]
This implies that
\[
\left| F_\lambda(t) \right| \le \| {\bf F}_\lambda(t) \| \le (1-\delta)^{\lfloor \lambda/(L+1) \rfloor} 
\| {\bf F}_0(t) \| \le C\, e^{-\eta \lambda}
\]
holds uniformly for all $t\in \R$.
\end{proof}

\section{Proof of Theorem~\ref{Th2}}\label{secTh2}

The essential step in the proof of Theorem~\ref{Th2} is the application of a theorem by the second and third authors~\cite[Theorem~1]{MR2018}.

Assume that $g$ is a non-periodic strongly $b$-multiplicative function of modulus $1$.
By our Proposition~\ref{prp_fourierprop}, the function 
$f$ defined by $f(n)=g(p^2n)\overline{g(q^2n)}$ belongs to the set
$\mathcal F_{\gamma,c}$ defined in~\cite[Definition~4]{MR2018}, where $c>0$ is arbitrary and $\gamma(\lambda)$ is maximal such that $Cb^{-\eta\lambda}\leq b^{-\gamma(\lambda)}$ for all $\lambda\geq 0$. (here $C$ and $\eta$ are as in Proposition~\ref{prp_fourierprop}).
Clearly, $\gamma(\lambda)\gg \eta\lambda$.

In order to apply Theorem~1 from~\cite{MR2018}, it is therefore sufficient to verify a \emph{carry property}~\cite[Definition~3]{MR2018} for the function $f$.
For this, we define, for any function $h:\mathbb N\rightarrow\mathbb C$ and $\lambda\geq 0$, the \emph{truncation} $h_\lambda$ as the 
$b^\lambda$-periodic continuation of $h\mid [0,b^\lambda)$. 
This function only takes into account the digits with indices below $\lambda$.

\begin{lemma}
Assume that $g$ is a non-periodic strongly $b$-multiplicative function of modulus $1$.
Define $f(n)=g(p^2n)\overline{g(q^2n)}$.
There exists $C>0$ such that for all nonnegative integers $\lambda,\kappa,\rho$ satisfying $\rho<\lambda$,
the number of integers $0\leq \ell<b^\lambda$ such that
\begin{equation}\label{eqn_carry}
f(\ell b^\kappa+k_1+k_2)\overline{f(\ell b^\kappa+k_1)}\neq
f_{\kappa+\rho}(\ell b^\kappa+k_1+k_2)\overline{f_{\kappa+\rho}(\ell b^\kappa+k_1)}
\end{equation}
for some $(k_1,k_2)\in\{0,\ldots,b^\kappa-1\}^2$ is bounded by
$Cb^{\lambda-\rho}$.
\end{lemma}
\begin{proof}
Separating the factors corresponding to $p$ and $q$, it is sufficient to verify this property for the function $f(n)=g(an)$, where $a\geq 0$.
We need to investigate the carry propagation occurring in the addition
$s_1+s_2$, where $s_1=a\ell b^\kappa+ak_1$ and $s_2=ak_2$. 
If $s_1\in [0,b^{\kappa+\rho}-ab^\kappa)+b^{\kappa+\rho}\mathbb N$,
the addition of $s_2$ does not change the base-$b$ digits of $s_1$ above $\kappa+\rho$; it is therefore sufficient to demand that 
$a\ell\in [0,b^\rho-2a)+b^\rho\mathbb N$ 
in order to obtain equality in~\eqref{eqn_carry} for all $k_1,k_2$.
For $\rho$ large enough, this condition is violated for
$\mathcal O(b^{\lambda}a/b^\rho)$ many $\ell<b^\lambda$, which implies the statement.
\end{proof}

Applying Mauduit and Rivat's Theorem~1~\cite{MR2018}, we obtain
\[\sum_{0\leq n<N}g(p^2n^2)\overline{g(q^2n^2)}=o(N)\]
for all strongly $b$-multiplicative functions $g$ of modulus $1$ and coprime $p$ and $q$ that are also coprime to $b$.
As a final step, we apply the Daboussi--K\'atai criterion~\cite{BSZ2013,K1986}
\begin{lemma}[Daboussi--K\'atai/Bourgain--Sarnak--Ziegler]
  \label{lemma:daboussi-katai}
Let $f:\mathbb N\rightarrow\mathbb C$ be bounded and such that
\begin{equation}\label{eqn_DK_sufficient}
\sum_{n \leq x} f(pn) \overline{f(qn)} = o(x)
\end{equation}
for all distinct primes $p$ and $q$. Then
\[\sum_{n \leq x} \mu(n)f(n) =o(x).\]
\end{lemma}
In fact it is sufficient to restrict the condition~\eqref{eqn_DK_sufficient} to large enough primes $p$ and $q$ --- Bourgain--Sarnak--Ziegler~\cite[page~80]{BSZ2013} note that their proof only involves primes larger than an arbitrary bound.
Applying this lemma to $f(n)=g(n^2)$, we obtain
\[\sum_{0\leq n<N}\mu(n)g(n^2)=o(N)\]
and therefore our Theorem~\ref{Th2}.

\subsection*{Acknowledgements}
The authors wish to thank Mariusz Lema\'nczyk for pointing out valuable references to the literature concerning M\"obius orthogonality of positive entropy dynamical systems, and helping us place our result into the context of related research works.
Moreover, we thank Clemens M\"ullner for several fruitful discussions.
%
%
%
%
%
%
%

\bibliographystyle{siam}
\bibliography{DMRS}
\bigskip
\end{document}